\documentclass[a4paper,12pt]{article}
\usepackage{relsize}
\usepackage{cite}
\usepackage{color}
\usepackage{mathtools}
\usepackage[dvipsnames]{xcolor}

\usepackage{enumitem}

\usepackage{amsmath,amsthm,amssymb,mathrsfs}

\usepackage[margin=2.5cm]{geometry}

\usepackage{graphicx}
\usepackage[T1]{fontenc}
\usepackage{authblk}
\usepackage[english]{babel}
\usepackage{stackrel}

\usepackage{array,booktabs,longtable}
\theoremstyle{plain}
\newtheorem{theorem}{Theorem}[section]
\newtheorem{lemma}{Lemma}[section]
\newtheorem{corollary}{Corollary}[section]

\theoremstyle{remark}
\newtheorem{remark}{Remark}[section]
\newtheorem*{remark*}{Remark}
\newtheorem{question}{Question}[section]

\numberwithin{equation}{section} 

\newcommand{\F}{\mathbb F}

\newcommand{\Z}{\mathbb Z}
\newcommand{\one}{\mathbf 1}

\title{{\bf Proofs and disproofs of some conjectures on zero-sum problems of new types}}

\author{Yang Li\thanks{Yang Li is with the School of Physical and Mathematical Sciences,
Nanyang Technological University, 21 Nanyang Link,
Singapore 637371, Singapore (email:~ yanglimath@163.com).},~
San Ling\thanks{San Ling is with the School of Physical and Mathematical Sciences,
Nanyang Technological University, 21 Nanyang Link,
Singapore 637371, Singapore (email:~ lingsan@ntu.edu.sg). 
He is also with VinUniversity, Vinhomes Ocean Park, Gia Lam,
Hanoi 100000, Vietnam (email:~ ling.s@vinuni.edu.vn).},~
Zhonghua Sun\thanks{Zhonghua Sun is with the School of Mathematics, Hefei University of Technology,
Hefei 230601, Anhui, China (email:~ sunzhonghuas@163.com).},~
Shixin Zhu\thanks{Shixin Zhu is with the School of Mathematics, Hefei University of Technology,
Hefei 230601, Anhui, China (email:~ zhushixin@hfut.edu.cn).}
}

\begin{document}
\maketitle

\begin{abstract}
For integers $r\ge1$ and $n\ge2$, let $s_r(n)$ be the least $N$
such that every sequence of $N$ vectors in $\Z^r$, each nonzero
modulo $n$, contains $n$ terms whose sum is zero modulo $n$
but nonzero modulo $n^2$.
Let $w_r(n)$ be the corresponding least $N$ when any nonempty
subsequence is allowed.
These constants were introduced as zero-sum problems of new types by Z.-W. Sun.  
We prove $s_r(n)\ge(2^{r+1}-2)(n-1)+1$ by constructing an explicit sequence from a lattice polytope, 
\emph{disproving} his conjecture $s_2(n)=4n+1$ for $n\ge4$.
We obtain $w_r(2)=2r+1$ by classifying all sequences avoiding
such nonempty subsequences modulo $4$.
We also \emph{prove} Sun's conjecture $w_2(n)=4n-3$ for
$n=2^k3^\ell p$, where $p$ is a prime and $k,\ell\ge0$ are integers, 
by classifying the extremal sequences of length $4p-4$
modulo $p^2$ when $n=p$.
For every fixed $n\ge2$, we prove 
$w_r(n)/s(\Z_n^r)=O_n(r^2/2^r)\to0$ as $r\to\infty$,
where $s(\Z_n^r)$ is the classical Erd\H{o}s-Ginzburg-Ziv
constant, with an explicit threshold on $r$ for strict inequality.
Thus, for every fixed modulus $n\ge2$, Sun's proposed equality fails 
for sufficiently large $r$.
\end{abstract}



\section{Introduction}

Let $r\ge1$ and $n\ge2$ be integers, and let
$A=\mathbf{a}_1\cdots \mathbf{a}_N$ be a sequence of length $N$ in $\Z^r$,
where $\Z$ denotes the set of integers.
Put $\Z_n=\Z/n\Z$.
For a prime $p$, let $\F_p=\Z_p$ denote the field with $p$ elements.
Write $G=\Z_n^r$, a finite abelian group of \emph{rank} $r$,
and $[m]=\{1,\ldots,m\}$ for $m\ge1$.
All vector congruences are coordinatewise, and $\mathbf{0}$
denotes a zero vector of an appropriate length.
Recall that the classical Erd\H{o}s-Ginzburg-Ziv constant $s(G)$
is the least $N$ such that every sequence over $G$ of length $N$
contains a zero-sum subsequence of length $n$ \cite{GG06}.
In terms of integer vectors, this requires an index set
$I\subseteq[N]$ with $|I|=n$ satisfying
\begin{equation}\label{eq:egz}
 \sum_{i\in I}\mathbf{a}_i\equiv\mathbf{0}\pmod n.
\end{equation}
The well-known Erd\H{o}s-Ginzburg-Ziv theorem \cite{EGZ61} gives $s(\Z_n)=2n-1$ for $r=1$,
while Kemnitz's conjecture \cite{K83}, proved by Reiher \cite{R07}, gives $s(\Z_n^2)=4n-3$ for $r=2$.
For $r\ge3$, bounds for $s(\Z_n^r)$ have been deeply studied
in~\cite{EEGKR07,H18,N20}.

Note that the condition in \eqref{eq:egz} depends only on residues modulo $n$
and does not determine whether the selected sum also vanishes modulo $n^2$.
For prime moduli, Gao et al.~\cite{GHJLW25}
studied exact divisibility conditions for zero-sum subsequences.
For a prime $p$, the \emph{$p$-adic valuation} of an integer $x$ is
$v_p(x)=\sup\{k\in\Z:~k\ge 0\text{ and }p^k\mid x\}$.
For $r=1$ and $n=p$, their constant $s(p,1)$ coincides with
the least length such that every integer sequence of that length whose terms
are not divisible by $p$ contains a $p$-term subsequence
whose sum has $p$-adic valuation exactly $1$.
Equivalently, the selected sum must vanish modulo $p$
but not modulo $p^2$. 
Recently, it was shown in \cite{GJLLY26} that $s(p,1)\le3p-2$
for every prime $p$, improving the bound
$s(p,1)\le s(\Z_p^2)=4p-3$ obtained from Kemnitz's theorem \cite{K83,R07}.

Very recently, Sun~\cite{S26} generalized the above rank-one prime-modulus problem
of Gao et al. \cite{GHJLW25} to arbitrary ranks $r\ge1$ and moduli $n\ge 2$
with respect to prescribed or unrestricted subsequence lengths.
We call a sequence $A$ \emph{admissible} if
$\mathbf{a}_i\not\equiv\mathbf{0}\pmod n$ for each $i$.
Following Sun's Definitions~1.2 and~1.3,
let $w_r(n)$ be the least $N$ such that every admissible $A$
of length $N$ has a nonempty $I\subseteq[N]$ satisfying
\begin{equation}\label{eq:required-sum}
 \sum_{i\in I}\mathbf{a}_i\equiv\mathbf{0}\pmod n,
 \quad\text{but}\quad
 \sum_{i\in I}\mathbf{a}_i\not\equiv\mathbf{0}\pmod{n^2}.
\end{equation}
Let $s_r(n)$ be the corresponding least length when $|I|=n$
is additionally required.
The problems of determining and bounding $s_r(n)$ and $w_r(n)$
for arbitrary $r\ge1$ and $n\ge2$ are referred to as \emph{zero-sum problems of new types}.

By definition, we immediately have $w_r(n)\le s_r(n)$ and $s_1(p)=s(p,1)$.
For $n>2$, Sun proved that $2rn+1\le s_r(n)\le n(n^r-1)+1$ in \cite[Theorem~1.3]{S26}.
For $r=2$, Sun conjectured in the same paper that
$w_2(n)=4n-3$ for every $n>1$ in Conjecture~1.1(ii)
and $s_2(n)=4n+1$ for every $n>2$ in Conjecture~1.2.
In Remark~1.2, he also asked whether $w_r(n)=s(\Z_n^r)$
for all integers $n>1$ and $r>0$.
More related constants and conjectures can also be found in \cite{GHJLW25,S26}.

This paper aims to study the zero-sum problems of new types and
to answer Sun's conjectures and questions listed above.
In the following, we call an admissible sequence \emph{bad}
if no $n$-term subsequence satisfies~\eqref{eq:required-sum}, and
\emph{$w$-bad} if no nonempty subsequence does.
Such sequences give lower bounds for $s_r(n)$ and $w_r(n)$, respectively.
We write $|A|$ for the length of $A$,
$\sigma(A)$ for the sum of its terms in the ambient abelian group, and
$\mathbf{v}^{[m]}$ for $m$ copies of $\mathbf{v}$.
In addition, sequence products mean concatenation and an empty product denotes the empty sequence.

Our first main result addresses Sun's fixed-length conjecture 
through a construction from lattice polytopes. 

\begin{theorem}\label{thm:geometric}
Let $r\ge1$ and $n>1$.
For the standard basis $\mathbf{e}_1,\ldots,\mathbf{e}_r$ of $\Z^r$, the admissible sequence
\begin{align}\label{eq.A_rn}
A_{r,n}=\prod_{\mathbf{v}\in V_r}\mathbf{v}^{[n-1]},~{\rm with}~
V_r=\{\pm\one_J:~\varnothing\ne J\subseteq[r]\}~{\rm and}~ \one_J=\sum_{j\in J}\mathbf{e}_j
\end{align}
is bad. 
Moreover, we have
\begin{equation}\label{eq:uniform-lower-bound}
 s_r(n)\ge(2^{r+1}-2)(n-1)+1.
\end{equation}
\end{theorem}

\begin{remark}\label{rem.main1}
Our lower bound~\eqref{eq:uniform-lower-bound} indeed yields
  the following improvements over Sun's lower bounds~\cite{S26}.
  \begin{itemize}
    \item In \cite[Theorem~1.3(ii)]{S26}, Sun proved $s_r(n)\ge2rn+1$ for $r\ge1$ and
    $n>2$.
    Our bound is strictly stronger for $r\ge2$ and $n>2$, except
    $(r,n)=(2,3)$, where the bounds coincide.
    For $r=2$ and $n\ge4$, the bound~\eqref{eq:uniform-lower-bound}
    gives $$s_2(n)\ge6n-5>4n+1,$$ 
    \emph{disproving} Sun's rank-two conjecture
    \cite[Conjecture~1.2]{S26}.

    \item For $r\ge2$, Sun also proved
    $3\cdot2^r+1\le s_r(3)\le2\cdot3^r-1$ in \cite[Theorem~1.4(i)]{S26}.
    Our bound $s_r(3)\ge4\cdot2^r-3$ improves the lower bound by
    $2^r-4$, with strict improvement for $r\ge3$.
  \end{itemize}
\end{remark}

Our second main result presents some exact values for $w_r(n)$ and solves Sun's rank-two conjecture for certain moduli, 
which derives from a complete classification of all $w$-bad sequences of arbitrary lengths for $n=2$ and $r\geq 1$, 
and a complete classification of all $w$-bad sequences of length $4p-4$ modulo $p^2$ for $n=p$ and $r=2$. 
During the process, we also obtain doubling and tripling inequalities to propagate the result $w_2(t)$ to $w_2(dt)$ for $d\in\{2,3\}$.

\begin{theorem}\label{thm:exact}
The following statements hold.
\begin{enumerate}
\item[(i)] For every integer $r\ge1$, we have $w_r(2)=2r+1$.
\item[(ii)] For $d\in \{2,3\}$ and every integer $t>1$, we have $w_2(dt)\leq dw_2(t)+3(d-1)$.
\item[(iii)] For every prime $p$ and integers $k,\ell\ge0$, we have 
$w_2(2^k3^\ell p)=4\cdot2^k3^\ell p-3$. 
\end{enumerate}
\end{theorem}

\begin{remark}\label{rem.main2} 
Parts (i) and (ii) of Theorem~\ref{thm:exact} play key roles 
in the proof of Theorem \ref{thm:exact}(iii). 
For $n=2^k3^\ell p$ with $p$ being a prime and $k,\ell\ge0$ being any integers, 
Theorem~\ref{thm:exact}(iii) confirms Sun's rank-two
conjecture~\cite[Conjecture~1.1(ii)]{S26}. 
The equality $w_r(n)=s(\Z_n^r)$ asked about in
\cite[Remark~1.2]{S26} fails in general, but holds for certain parameters.  
    \begin{itemize}
    \item Since $s(\Z_2^r)=2^r+1$~\cite{H73},
Theorem~\ref{thm:exact}(i) answers this question
\emph{negatively} for every $r\ge3$. 

    \item Since $s(\Z_n^2)=4n-3$ for every $n\ge2$
by~\cite{K83,R07},
Theorem~\ref{thm:exact}(iii) answers this question
\emph{affirmatively} in rank two for these moduli.
\end{itemize}
\end{remark}

Our third main result shows that this failure of the equality in Sun's question, 
as stated in Remark \ref{rem.main2}, 
is not confined to the binary case.

\begin{theorem}\label{thm:all-moduli}
For integers $n\ge2$ and $r\ge1$, we have
$$
 w_r(n)\le n\lceil r\log_2 n\rceil
 \bigl(\lceil r\log_2 n\rceil+1\bigr)+1, 
$$ 
and 
$$
  \frac{w_r(n)}{s(\Z_n^r)}
  =O_n\!\left(\frac{r^2}{2^r}\right)\longrightarrow0,~ r\to\infty, 
$$
where $O_n$ denotes big-$O$ notation with an implied constant
that may depend on $n$ but is independent of $r$. 
Moreover, $w_r(n)<s(\Z_n^r)$ 
whenever $r\geq 3$ for $n=2$ and 
$r\geq \left\lceil4\log_2(1+\log_2 n)+4\right\rceil$ for $n\geq 3$.  
\end{theorem}

\begin{remark}\label{rem.main3}
Theorem~\ref{thm:all-moduli} answers Sun's equality
question~\cite[Remark~1.2]{S26} negatively for every fixed modulus $n$ 
and all sufficiently large ranks $r$. 
In particular, the limit strengthens this conclusion to an asymptotic separation
between $w_r(n)$ and the classical Erd\H{o}s-Ginzburg-Ziv constant $s(\Z_n^r)$.
\end{remark}

The rest of the paper is organized as follows.
Section~\ref{sec:lower} proves Theorem~\ref{thm:geometric}.
Section~\ref{sec:exact} gives the binary and prime-modulus
classifications and proves Theorem~\ref{thm:exact}. 
Section~\ref{sec:all-moduli} proves Theorem~\ref{thm:all-moduli}.

\section{Proof of Theorem~\ref{thm:geometric}}\label{sec:lower}

We begin with a construction of admissible bad sequences from
lattice polytopes.
A \emph{lattice of rank $r$} is the integer span of $r$ linearly
independent real vectors. Let $L$ be such a lattice with basis
$\mathbf{b}_1,\ldots,\mathbf{b}_r$.
Its real linear span, denoted by $\operatorname{span}_{\mathbb R}L$,
consists of all real linear combinations of these basis vectors.
For a set $S$ of real vectors, its \emph{convex hull} is
$
\operatorname{conv}(S)=
\{\sum_{i=1}^{m}\lambda_i\mathbf{x}_i:~
m\ge1,\ \mathbf{x}_i\in S,\ \lambda_i\in\mathbb R,\ \lambda_i\ge0,
\sum_{i=1}^{m}\lambda_i=1\}.
$
A \emph{lattice polytope in $L$}, denoted by $P$, is the convex hull of a finite
nonempty subset of $L$.
Its \emph{dimension} is defined by the dimension of the real linear
span of all differences $\mathbf{x}-\mathbf{y}$ with
$\mathbf{x},\mathbf{y}\in P$.
Therefore, $\dim(P)\le r$, and $\dim(P)=r$ if and only if
the differences span $\operatorname{span}_{\mathbb R}L$.

A \emph{vertex} of $P$ is a point $\mathbf{v}\in P$ that does not lie strictly
between two distinct points of $P$.
We denote the set of all vertices by $\operatorname{Vert}(P)$.
A point $\mathbf{y}\in P$ is an \emph{interior point}
relative to $\operatorname{span}_{\mathbb R}L$ if, for some
$\varepsilon>0$, we have
$
\{\mathbf{y}+\sum_{i=1}^r t_i\mathbf{b}_i:~
t_i\in\mathbb R,\ |t_i|<\varepsilon,~ 1\le i\le r\}
\subseteq P.
$
Note also that an interior point cannot be a vertex in a lattice polytope of positive dimension. 

\begin{lemma}\label{prop:lattice-bounds}
Let $r\ge1$ and $n>1$, and let $P$ be a lattice polytope of dimension $r$
in a lattice $L$ of rank $r$ such that $\mathbf{0}$ is an interior point
of $P$ relative to $\operatorname{span}_{\mathbb R}L$ and
$P\cap L=\{\mathbf{0}\}\cup\operatorname{Vert}(P)$.
The admissible sequence
$$
A(P)=\prod_{\mathbf{v}\in\operatorname{Vert}(P)}
\mathbf{v}^{[n-1]}
$$
is bad. In particular,
$
s_r(n)\ge(n-1)|\operatorname{Vert}(P)|+1.
$
\end{lemma}
\begin{proof}
Fix a lattice basis $\mathbf{b}_1,\ldots,\mathbf{b}_r$
and identify $L$ with $\Z^r$ via the bijection
$$
\sum_{i=1}^r m_i\mathbf{b}_i
\longleftrightarrow(m_1,\ldots,m_r),~{\rm with}~ m_i\in\Z.
$$
Therefore, any point $\mathbf{x}\in L$ corresponds to a unique coordinate vector $(m_1,\ldots,m_r)$.
For $k=1,2$, we write $n^kL=\{n^k\mathbf{x}:~\mathbf{x}\in L\}$.
Then a point belongs to $n^kL$ if and only if all its coordinates
are divisible by $n^k$.
Thus, the bijection maps $n^kL$ onto $n^k\Z^r$.
Since it is also additive, it preserves admissibility and
\eqref{eq:required-sum}.

For any vertex $\mathbf{v}$, the conditions $\dim(P)=r\ge1$
and $\mathbf{0}$ interior imply $\mathbf{v}\ne\mathbf{0}$.
Convexity and the interior condition also give $\varepsilon>0$
such that, for real $t_i$ with $|t_i|<\varepsilon$ for all $i$, we have
$
\frac{\mathbf{v}}{n}
+\left(1-\frac{1}{n}\right)\sum_{i=1}^r t_i\mathbf{b}_i\in P.
$
In particular, taking $t_i=0$ for all $i$ gives $\frac{\mathbf{v}}{n}\in P$.
As $(1-\frac{1}{n})\varepsilon>0$, the point $\frac{\mathbf{v}}{n}$ is interior.
If $\mathbf{v}$ belonged to $nL$, then $\frac{\mathbf{v}}{n}\in L$.
In this case, we deduce that $\mathbf{0}\ne \frac{\mathbf{v}}{n}\in P\cap L$ would be an interior point,
contrary to the assumption $P\cap L=\{\mathbf{0}\} \cup \operatorname{Vert}(P)$.
Therefore, $\mathbf{v}\notin nL$.
This shows that $A(P)$ is admissible.

Let $\mathbf{a}_1,\ldots,\mathbf{a}_n$ be any selected terms of $A(P)$
with sum $\mathbf{x}\in nL$.
By convexity, $\frac{\mathbf{x}}{n}\in P$, and the assumption
$\mathbf{x}\in nL$ gives $\frac{\mathbf{x}}{n}\in L$.
Thus, $\frac{\mathbf{x}}{n}\in P\cap L$.
Suppose that $\frac{\mathbf{x}}{n}$ is a vertex.
For each $1\le i\le n$,
we have $$
\frac{\mathbf{x}}{n}
=\frac{1}{n}\mathbf{a}_i
+\frac{n-1}{n}
 \left(\frac{1}{n-1}\sum_{j\ne i}\mathbf{a}_j\right)\in P,
 ~{\rm and}~\frac{1}{n-1}\sum_{j\ne i}\mathbf{a}_j\in P. 
$$
Since $n>1$, both $\frac{1}{n}$ and $\frac{n-1}{n}$ are positive.
By the definition of a vertex, we conclude that
$\mathbf{a}_i=\frac{\mathbf{x}}{n}$ for every $1\le i\le n$.
However, each vertex occurs only $n-1$ times in $A(P)$, a contradiction.
Hence, $\frac{\mathbf{x}}{n}$ is not a vertex.
It then follows from the assumption on $P\cap L$ that $\frac{\mathbf{x}}{n}=\mathbf{0}$,
and hence, $\mathbf{x}=\mathbf{0}\in n^2L$.
As a result, $A(P)$ is bad and
$$s_r(n)\ge |A(P)|+1=(n-1)|\operatorname{Vert}(P)|+1.$$
This completes the proof.\end{proof}

We use $[x,y]=\{t\in\mathbb R:~x\le t\le y\}$ to denote a closed interval of real numbers.
The sets $[0,1]^r$ and $[-1,0]^r$ consist of real vectors with coordinates in $[0,1]$ and $[-1,0]$, respectively.
To prove Theorem~\ref{thm:geometric}, we apply
Lemma~\ref{prop:lattice-bounds} with 
\begin{align}\label{eq:polytope}
  \begin{split}
    L_r=\Z^r\quad{\rm and}\quad
    P_r  = \operatorname{conv}([0,1]^r\cup[-1,0]^r)
  \end{split}
\end{align}

\begin{proof}[\textbf{Proof of Theorem~\ref{thm:geometric}}]
We verify that $P_r$ in~\eqref{eq:polytope} satisfies the
hypotheses of Lemma~\ref{prop:lattice-bounds} with $L_r=\Z^r$.
Note that $[0,1]^r$ and $[-1,0]^r$ are the convex hulls of
$\{0,1\}^r$ and $\{-1,0\}^r$, respectively.
By the definition of $\one_J$ in~\eqref{eq.A_rn}, every nonzero
vector in $\{0,1\}^r$ is $\one_J$ for a unique nonempty
$J\subseteq[r]$.
Therefore,
$$
P_r=\operatorname{conv}(\{0,1\}^r\cup\{-1,0\}^r)
=\operatorname{conv}(V_r\cup\{\mathbf{0}\}).
$$
Since the set $V_r\cup\{\mathbf{0}\}\subseteq L_r$ is finite,
$P_r$ is a lattice polytope in $L_r$.
Since $\mathbf{0}$ is the midpoint of
$\mathbf{e}_1,-\mathbf{e}_1\in V_r$, we also have
$P_r=\operatorname{conv}(V_r)$.
Moreover, the points $\pm\mathbf{e}_i\in V_r$ for $1\le i\le r$
give
$
\left\{
\sum_{i=1}^r t_i\mathbf{e}_i:
t_i\in\mathbb R,\ \sum_{i=1}^r|t_i|\le1
\right\}\subseteq P_r.
$
Hence, $\dim(P_r)=r$ and $\mathbf{0}$ is an interior point of $P_r$ relative to $\operatorname{span}_{\mathbb R}L_r=\mathbb R^r$.

For every $\mathbf{x}=(x_1,\ldots,x_r)$ in $[0,1]^r$ or $[-1,0]^r$,
we have $|x_i|\le1$ and $|x_i-x_j|\le1$ for all $1\le i,j\le r$.
Since convex combinations preserve these inequalities,
they hold throughout $P_r$.
It implies that each integer point of $P_r$ has coordinates in $\{-1,0,1\}$,
and the difference bound $|x_i-x_j|\le1$ excludes the simultaneous occurrence
of coordinates equal to $1$ and $-1$.
Hence, such an integer point $\mathbf{x}=(x_1,\ldots,x_r)\in P_r\cap L_r$ is either $\mathbf{0}$ or belongs to $V_r$,
$i.e.$, $P_r\cap L_r\subseteq\{\mathbf{0}\}\cup V_r$.
Combining with the fact $\{\mathbf{0}\}\cup V_r\subseteq P_r\cap L_r$, we arrive at the conclusion
$$
P_r\cap L_r=\{\mathbf{0}\}\cup V_r.
$$

We then determine the vertices of $P_r$.
For a nonempty $J\subseteq[r]$ and a real vector $\mathbf{x}=(x_1,\ldots,x_r)\in\mathbb R^r$,
we define $k=|J|$ and
$$
\ell_J(\mathbf{x})
=(r-k+1)\sum_{i\in J}x_i-k\sum_{i\notin J}x_i.
$$
For every nonempty $K\subseteq[r]$, we then have
$$
\ell_J(\one_K)
=(r-k+1)|K\cap J|-k|K\setminus J|
\le k(r-k+1),
$$
with equality if and only if $K=J$.
For the negative vectors, we have
$$
\ell_J(-\one_K)
=-(r-k+1)|K\cap J|+k|K\setminus J|
\le k(r-k)<k(r-k+1).
$$
Since $P_r=\operatorname{conv}(V_r)$ and $\ell_J$ is linear,
these inequalities show that $\ell_J$ attains its unique maximum
on $P_r$ at $\one_J$.
If $\one_J$ lay strictly between two distinct points of $P_r$,
linearity would force both points to attain the same maximum.
Therefore, $\one_J$ is a vertex of $P_r$. 
The symmetry $P_r=-P_r$ gives the same conclusion for $-\one_J$.
Since every vertex of the convex hull of a finite set belongs to that set,
we deduce that
$$
\operatorname{Vert}(P_r)=V_r.
$$

Thus, $P_r$ satisfies the hypotheses of
Lemma~\ref{prop:lattice-bounds}, and $A(P_r)=A_{r,n}$.
Since $|V_r|=2(2^r-1)$, the lemma yields the desired badness and
the lower bound~\eqref{eq:uniform-lower-bound}.
\end{proof}

\begin{remark}\label{rem:lattice-bound-scope}
For a polytope $P$ as in
  Lemma~\ref{prop:lattice-bounds}, Nill's bound says that 
  $|\operatorname{Vert}(P)|\le2^{r+1}-2$ \cite[Corollary~6.3]{N05}. 
  Therefore, the polytope $P_r$ used in Theorem~\ref{thm:geometric} reaches this bound 
  and gives the largest lower bound on $s_r(n)$ obtainable from this lemma for fixed $r$ and $n$.
\end{remark}

\section{Proof of Theorem~\ref{thm:exact}}\label{sec:exact}
This section proves Theorem~\ref{thm:exact} in three parts.  
For some exact values of $w_r(n)$ determined in this section, 
we need to classify all $w$-bad sequences of the corresponding lengths.
Note also that analogous classification questions for classical 
zero-sum constants were studied as inverse zero-sum problems
in~\cite{GGS07,S12}.

\subsection{A lower bound on $w_r(n)$ and the proof of Theorem~\ref{thm:exact}(i)}

We first provide a unified lower bound on $w_r(n)$ by constructing a $w$-bad sequence 
in $\Z^r$ of length $2r(n-1)$ as follows.

\begin{lemma}\label{lem:axis}
Let $n>1$ and $r\ge1$. Then the admissible sequence 
$$E_{r,n}=\prod_{j=1}^r\mathbf{e}_j^{[n-1]}(-\mathbf{e}_j)^{[n-1]}$$  
is $w$-bad. Moreover, $w_r(n)\ge2r(n-1)+1$.
\end{lemma}
\begin{proof}
By construction, $E_{r,n}$ is admissible.
For a subsequence, let $u_j$ and $v_j$ count the selected copies of 
$\mathbf{e}_j$ and $-\mathbf{e}_j$, respectively. 
Since $0\le u_j,v_j\le n-1$, $u_j-v_j$ is divisible by $n$ only when $u_j=v_j$.
Therefore, every subsequence of $E_{r,n}$ with sum zero modulo $n$
has sum zero in $\Z^r$, and $E_{r,n}$ is $w$-bad.
Since $|E_{r,n}|=2r(n-1)$, we obtain $w_r(n)\ge 2r(n-1)+1$.
\end{proof}

We now prove Theorem~\ref{thm:exact}(i) by showing that
$2r+1$ is also an upper bound for $w_r(2)$.
For this purpose, the following theorem classifies all nonempty
$w$-bad sequences in $\Z^r$ for $n=2$, with the normal form
specified modulo $4$.

\begin{theorem}\label{thm:binary}
Let $A=\prod_{i=1}^N \mathbf{a}_i$ be a nonempty
sequence in $\Z^r$ with $\mathbf{a}_i\not\equiv\mathbf0\pmod2$
for every $i\in[N]$, and let
$
 d=\dim_{\F_2}\langle
 \overline{\mathbf{a}}_1,\ldots,\overline{\mathbf{a}}_N\rangle,
$
where a bar denotes reduction modulo $2$.
The sequence $A$ is $w$-bad if and only if, after reordering,
\begin{equation}\label{eq:binary-normal-form}
\begin{aligned}
 A=\mathbf{x}_1\cdots \mathbf{x}_d\, \mathbf{y}_1\cdots \mathbf{y}_k,~{\rm with}~
 \mathbf{y}_j\equiv-\sum_{i\in S_j}\mathbf{x}_i
     \pmod4 ~{\rm for}~1\le j\le k,
\end{aligned}
\end{equation}
where $k\ge0$, the $\overline{\mathbf{x}}_i$ are linearly independent,
and the nonempty subsets $S_j\subseteq[d]$ are pairwise disjoint.
\end{theorem} 
\begin{proof}
To prove \emph{necessity}, suppose that $A$ is $w$-bad.
Choose terms $\mathbf{x}_1,\ldots,\mathbf{x}_d$ whose residues
form a basis of the residue span of $A$.
After reordering, write
$A=\mathbf{x}_1\cdots\mathbf{x}_d\mathbf{y}_1\cdots\mathbf{y}_k$,
where $k=|A|-d$.
Since $\overline{\mathbf{y}}_j\ne\mathbf{0}$ for every $j\in[k]$,
there is a unique nonempty set $S_j\subseteq[d]$ such that
$\overline{\mathbf{y}}_j=\sum_{i\in S_j}\overline{\mathbf{x}}_i$.
By these residue expressions, both the subsequences
$\mathbf{y}_j\prod_{i\in S_j}\mathbf{x}_i$ for $j\in[k]$ and
$\mathbf{y}_j\mathbf{y}_\ell
\prod_{i\in S_j\mathbin\triangle S_\ell}\mathbf{x}_i$
for distinct $j,\ell\in[k]$
have sum zero modulo $2$, 
where $S_j\mathbin\triangle S_{\ell}= (S_j \setminus S_{\ell}) \cup (S_{\ell} \setminus S_j)$ 
denotes the \emph{symmetric difference}. 
Since $A$ is $w$-bad, their sums are also zero modulo $4$.
In particular, we have 
$
\mathbf{y}_j+\sum_{i\in S_j}\mathbf{x}_i
\equiv\mathbf{0}\pmod4
$
and
\begin{equation*}
\mathbf0
\equiv \mathbf y_j+\mathbf y_\ell
+\sum_{i\in S_j\triangle S_\ell}\mathbf x_i
\equiv-\sum_{i\in S_j}\mathbf x_i
-\sum_{i\in S_\ell}\mathbf x_i
+\sum_{i\in S_j\triangle S_\ell}\mathbf x_i
=-2\sum_{i\in S_j\cap S_\ell}\mathbf x_i
\pmod4,
\end{equation*}
which further implies   
$\sum_{i\in S_j\cap S_\ell}\overline{\mathbf{x}}_i=\mathbf{0}$ in $\F_2^r$.  
By the linear independence of $\overline{\mathbf{x}}_1,\ldots,\overline{\mathbf{x}}_d$ over $\F_2$, 
we deduce that $S_j\cap S_\ell=\varnothing$. 
As a result, $A$ has the form~\eqref{eq:binary-normal-form}
with all the conditions described in the statement.

\emph{Conversely}, suppose that $A$ has the form described in the statement.
Let $B$ be any nonempty subsequence of $A$ whose sum vanishes modulo $2$,
and let $J\subseteq[k]$ index the terms $\mathbf{y}_j$ selected in $B$. 
Linear independence and the pairwise disjointness of the $S_j$
imply that $J\ne\varnothing$ and that the selected $\mathbf{x}_i$
are exactly those with $i\in\bigcup_{j\in J}S_j$. 
It then follows from \eqref{eq:binary-normal-form} that 
$$
\sigma(B)
=\sum_{j\in J}
(\mathbf{y}_j+\sum_{i\in S_j}\mathbf{x}_i)
\equiv\mathbf{0}\pmod4.
$$
Thus, $A$ is $w$-bad. This completes the proof. 
\end{proof}

\begin{proof}[\textbf{Proof of Theorem~\ref{thm:exact}(i)}]
Using the disjointness and nonemptiness of the sets
$S_j$ in Theorem~\ref{thm:binary}, we obtain
$k\le d\leq r$ and $|A|=d+k\le2d\le2r$ for every $w$-bad sequence $A$.
It follows that $w_r(2)\le2r+1$. 
On the other hand, taking $n=2$ in Lemma~\ref{lem:axis} yields $w_r(2)\geq 2r+1$. 
We therefore conclude that $w_r(2)=2r+1$ for every $r\ge1$. 
\end{proof}

\subsection{Classification of some $w$-bad sequences}

We classify, modulo $p^2$, all $w$-bad sequences
in $\Z^2$ of length $4p-4$ when $n=p$ is a prime and deduce
$w_2(p)=4p-3$. 
These results are needed for the proof of Theorem~\ref{thm:exact}(ii)  
in the next subsection.

For an integer vector sequence $J$, let $\Sigma(J)\subseteq\Z_n^r$
be the set of \emph{all} residues modulo $n$ obtained by selecting a
subsequence and summing its terms, including the empty sum $\mathbf{0}$. 
It is also a classical problem and an important task to determine the size of $\Sigma(J)$ 
(see, e.g., \cite{P09,M13,PLLH20}).
Given an integer $t$, we define the dilated set 
$t\Sigma(J)=\{t\mathbf{s}:~\mathbf{s}\in\Sigma(J)\}$.
For a cyclic subgroup $H$ of $\Z_n^r$, 
we denote by $[\Z_n^r:~H]=\frac{n^r}{|H|}$ the number of cosets 
$\mathbf{g}+H=\{\mathbf{g}+\mathbf{h}:~\mathbf{h}\in H\}$,
where $\mathbf{g}\in\Z_n^r$. 
We first provide some sufficient conditions for a sequence to be not $w$-bad.

\begin{lemma}\label{lem:two-complete-blocks}
Let $n>1$ and $r\ge1$, and let
$A=J_1J_2\mathbf{z}$ be admissible in $\Z^r$.
Let $\mathbf{b}=\mathbf{z}\bmod n$, and let
$H=\langle\mathbf{b}\rangle$ be the subgroup of $\Z_n^r$
generated by $\mathbf{b}$. 
Then $A$ is not $w$-bad if either of the following conditions holds: 
\begin{enumerate}
  \item [(i)] There exists a coset $C$ of $H$ such that
  $C\subseteq\Sigma(J_1)\cap(-\Sigma(J_2))$.
  \item [(ii)] $|\Z_n^r\setminus\Sigma(J_1)|+|\Z_n^r\setminus\Sigma(J_2)|<[\Z_n^r:~H]$. 
\end{enumerate}
\end{lemma} 
\begin{proof}
  (i) Assume that $A$ is $w$-bad.  
For each $\mathbf{g}\in C$, fix a subsequence $R$ of $J_2$
whose sum has residue $-\mathbf{g}$ modulo $n$ and define $f(\mathbf{g})=-\sigma(R)\bmod n^2$.  
For every subsequence $P$ of $J_1$ with
$\sigma(P)\equiv\mathbf{g}\pmod n$, the sum of $PR$
vanishes modulo $n$.
By $w$-badness of $A$, we arrive at the congruence 
\begin{equation}\label{eq.lemma33-common}
 \sigma(P)\equiv-\sigma(R)\equiv f(\mathbf{g})\pmod{n^2}.
\end{equation}

For any $\mathbf{g}\in C$, the coset property gives
$\mathbf{g}+\mathbf{b}\in C$.
Since $C\subseteq\Sigma(J_1)\cap(-\Sigma(J_2))$,
there are subsequences $P_+$ and $P$ of $J_1$, and a subsequence
$Q$ of $J_2$ such that 
$$
 \sigma(P_+)\equiv\mathbf{g}+\mathbf{b}\pmod n,~
 \sigma(P)\equiv\mathbf{g}\pmod n,~
 \sigma(Q)\equiv-\mathbf{g}-\mathbf{b}\pmod n.
$$
Then both $P_+Q$ and $PQ\mathbf{z}$ have sum zero modulo $n$. 
Combining with the $w$-badness of $A$ and \eqref{eq.lemma33-common}, we derive  
$
 \sigma(Q)+f(\mathbf{g}+\mathbf{b})
 \equiv \sigma(Q)+\mathbf{z}+f(\mathbf{g})
 \equiv\mathbf{0}\pmod{n^2}.
$
Therefore,
\begin{equation}\label{eq.lemma33}
 f(\mathbf{g}+\mathbf{b})
 \equiv f(\mathbf{g})+\mathbf{z}\pmod{n^2}.  
\end{equation}

Since $n\mathbf{b}=\mathbf{0}$ in $\Z_n^r$, 
it follows from applying \eqref{eq.lemma33} $n$ times that 
$$
 f(\mathbf{g})=f(\mathbf{g}+n\mathbf{b})
 \equiv f(\mathbf{g})+n\mathbf{z}\pmod{n^2}.
$$ 
It then turns out that $n\mathbf{z}\equiv\mathbf{0}\pmod{n^2}$, and hence, 
$\mathbf{z}\equiv\mathbf{0}\pmod n$, contrary to the admissibility of $\mathbf{z}$. 
As a result, condition~(i) implies that $A$ is not $w$-bad. 

(ii) It suffices to verify that condition~(i) can be derived from condition~(ii).
Note that every coset of $H$ that is not contained in
$\Sigma(J_1)\cap(-\Sigma(J_2))$ contains a residue outside this
intersection. Since the cosets partition $\Z_n^r$ and
$
 |\Z_n^r\setminus(-\Sigma(J_2))|
 =|\Z_n^r\setminus\Sigma(J_2)|,
$
the number of such cosets is at most
$$
 |\Z_n^r\setminus\Sigma(J_1)|
 +|\Z_n^r\setminus\Sigma(J_2)|< [\Z_n^r:H] 
$$ 
by condition~(ii).
Hence, there exists a coset $C$ of $H$ such that $C\subseteq\Sigma(J_1)\cap(-\Sigma(J_2))$, 
$i.e.$, condition~(i) holds. 
\end{proof}

The following two lemmas give estimates for the number of subsequence sums 
in $\Z_m$ with $m>1$ and $\F_p^2$ with $p\geq 3$, respectively. 
The first one generalizes Matom\"aki's bound \cite[Lemma~1]{M13} for prime moduli
to arbitrary finite cyclic groups. 
It can also be deduced by by taking $B=\{0,u\}$ in Chowla's sumset theorem~\cite{C35} 
and applying the bound successively, 
where $u$ is a generator. 
We include an elementary proof for completeness. 

\begin{lemma}\label{lem:cyclic-growth}
Let $u$ be a generator of the cyclic group $\Z_m$ with $m>1$. 
Then every nonempty proper subset $S$ of $\Z_m$ satisfies 
$$
 |S\cup(S+u)|>|S|, 
$$
where $S+u=\{s+u:~s\in S\}$. 
Moreover, if every term of a sequence $J$ of length $t\ge0$
generates $\Z_m$, then
$$
 |\Sigma(J)|\ge\min(m,t+1).
$$
\end{lemma}
\begin{proof}
Since $S\subseteq \Z_m$ is finite, 
we have $S+u=S$ if $|S\cup(S+u)|=|S|$. 
Since $S$ is nonempty and $u$ generates $\Z_m$, 
we further deduce that $S=\Z_m$, contrary to $S$ being proper.
Hence, $|S\cup(S+u)|>|S|$.

For the further assertion, the case $t=0$ first follows from 
$\Sigma(J)=\{0\}$ and $\min(m,0+1)=1$. 
For $t\ge1$, we write $J=u_1\cdots u_t$, where each $u_i$ generates
$\Z_m$ by the given condition. For every $1\le i\le t$, we have the iteration 
$$
 \Sigma(u_1\cdots u_i)
 =\Sigma(u_1\cdots u_{i-1})
  \cup\bigl(\Sigma(u_1\cdots u_{i-1})+u_i\bigr).
$$
Since each $u_i$ generates $\Z_m$, the inequality
$|S\cup(S+u)|>|S|$ shows that the number of subsequence
sums increases by at least one at each step, until either
all $m$ residues have been obtained or all $t$ terms have been added.
Since $\Sigma(\emptyset)=\{0\}$, this gives
$|\Sigma(J)|\ge\min(m,t+1)$. 
\end{proof}

\begin{lemma}{\rm (\!\! \cite[Lemma~15]{M13})} 
\label{lem:matomaki}
Let $p$ be a prime, and let $S$ be a sequence of $2p-2$
nonzero vectors in $\F_p^2$. If each one-dimensional subspace
contains at most $p-1$ terms of $S$, counted with multiplicity,
then
$
 |\Sigma(S)|\ge p^2-1.
$
\end{lemma}

Applying Lemma~\ref{lem:matomaki} to a sequence of length $2p-3$ 
gives the following improved result. 

\begin{lemma}\label{lem:two-missing}
Let $p\ge3$ be a prime. If a sequence $J$ in $\F_p^2$ has $2p-3$
nonzero terms and at most $p-2$ terms on each one-dimensional
subspace, then
$$
 |\F_p^2\setminus\Sigma(J)|\le2 
$$ 
and the bound is attained. 
\end{lemma}
\begin{proof}
For the upper bound, zero terms may be deleted without
changing $\Sigma(J)$, so we may assume $|J|=2p-3$. 
For every nonzero $\mathbf x\in\F_p^2$, we then have 
$|J\mathbf x|=2p-2$, with at most $p-1$ terms on each one-dimensional subspace. 
Since $\Sigma(J\mathbf x)=\Sigma(J)\cup(\Sigma(J)+\mathbf x)$,
we have
$$
 \F_p^2\setminus\Sigma(J\mathbf x)=(\F_p^2\setminus\Sigma(J)) \cap((\F_p^2\setminus\Sigma(J))+\mathbf x).
$$ 
Applying Lemma~\ref{lem:matomaki} to $J\mathbf x$ yields $|\Sigma(J\mathbf x)|\ge p^2-1$, and hence, 
\begin{small}
\begin{align}\label{eq.lem35}
 \bigl|\{(\mathbf b,\mathbf c)\in (\F_p^2\setminus\Sigma(J))^2:
            \mathbf b-\mathbf c=\mathbf x\}\bigr|
 =|(\F_p^2\setminus\Sigma(J))\cap((\F_p^2\setminus\Sigma(J))+\mathbf x)|\le1.  
\end{align}
\end{small}
Note that $J\setminus B$ can take all subsequences of $J$ as $B$ runs over all subsequences of $J$  
and $\sigma(J\setminus B)=\sigma(J)-\sigma(B)$. 
Then $\Sigma(J)=\sigma(J)-\Sigma(J)$  
and $\F_p^2\setminus\Sigma(J)=\sigma(J)-(\F_p^2\setminus\Sigma(J))$.
For distinct $\mathbf b,\mathbf c\in (\F_p^2\setminus\Sigma(J))$, 
since the pairs $(\mathbf b,\mathbf c)$ and $(\sigma(J)-\mathbf c,\sigma(J)-\mathbf b)$
have the same nonzero difference, we deduce from \eqref{eq.lem35} that  
$
 (\mathbf b,\mathbf c)
 =(\sigma(J)-\mathbf c,\sigma(J)-\mathbf b).
$ 
It implies that $(\F_p^2\setminus\Sigma(J))\setminus\{\mathbf b\}\subseteq\{\sigma(J)-\mathbf b\}$
for each $\mathbf b\in (\F_p^2\setminus\Sigma(J))$.   
Therefore, $|\F_p^2\setminus\Sigma(J)|\le2$. 

For the attainment of the bound, 
we construct the sequence 
$$J=\mathbf e_1^{[p-2]}\mathbf e_2^{[p-2]}(\mathbf e_1+\mathbf e_2),$$  
where $\mathbf e_1,\mathbf e_2$ form a basis of $\F_p^2$. 
Indeed, $\F_p^2\setminus\Sigma(J)=\{-\mathbf e_1,-\mathbf e_2\}$.  
\end{proof}

We now classify $w$-bad sequences in $\Z^2$ of length $4p-4$,
where $p$ is a prime.
For odd $p$, after a suitable change of coordinates, we split
the sequence into two subsequences of length $2p-2$.
We show that the first coordinates of one subsequence and the
second coordinates of the other form two $w$-bad integer sequences,
to which we apply the rank-one classification restated below.

\begin{lemma}{\rm (\!\! \cite[Theorem~1.2]{S26})}
\label{lem:sun-scalar-classification}
Let $n>1$ be an integer, and let
$A=a_1\cdots a_{2n-2}$ be a sequence in $\Z$ with
$\gcd(a_i,n)=1$ for every $1\le i\le2n-2$.
Then $A$ is $w$-bad if and only if, after reordering,
$
 A\equiv u^{[n-1]}(-u)^{[n-1]}\pmod{n^2}
$
for some $u\in\Z$ with $\gcd(u,n)=1$.
\end{lemma}

\begin{theorem}\label{lem:prime-w2}
Let $p$ be a prime and let $n=p$.
An admissible sequence $A$ in $\Z^2$ of length $4p-4$ is $w$-bad
if and only if, after reordering,
\begin{equation}\label{eq:prime-normal-form}
 A\equiv \mathbf{U}^{[p-1]}(-\mathbf{U})^{[p-1]}
          \mathbf{V}^{[p-1]}(-\mathbf{V})^{[p-1]}\pmod{p^2}
\end{equation}
for some $\mathbf{U},\mathbf{V}\in\Z^2$ with $\det(\mathbf{U},\mathbf{V})\not\equiv0\pmod p$,
where $\det(\mathbf{U},\mathbf{V})$ denotes the determinant of the $2\times2$ matrix with columns $\mathbf{U}$ and $\mathbf{V}$.
\end{theorem}
\begin{proof}
For \emph{necessity}, suppose that $A$ is $w$-bad.
If $p=2$, the desired result follows from Theorem~\ref{thm:binary} directly.
Henceforth, $p\ge3$.
For each one-dimensional subspace $L\le\F_p^2$, let $A_L$ be the
subsequence of terms of $A$ whose residues modulo $p$ lie in $L$.

We first show that $|A_L|=2p-2$ for some $L$.
For any fixed $L$, we further assume without loss of generality that
$L=\{(t,0):~t\in\F_p\}$, since invertible coordinate changes
modulo $p^2$ preserve admissibility and $w$-badness.
Let $\sigma(T)_i$ denote the $i$th coordinate of $\sigma(T)$.
Since $\sigma(T)_2\equiv0\pmod p$ for every
nonempty subsequence $T$ of $A_L$,
the $w$-badness of $A$ gives the deduction link 
\begin{equation}\label{eq:first-coordinate-lift}
 \sigma(T)_1\equiv0\pmod p
 \ \Longrightarrow\ \sigma(T)\equiv\mathbf0\pmod p
 \ \Longrightarrow\ \sigma(T)\equiv\mathbf0\pmod{p^2}.
\end{equation}
The first coordinates therefore form an admissible integer 
$w$-bad sequence. Combining with the equality
$w_1(p)=2p-1$ established in \cite[Theorem~1.1]{S26}, we have $|A_L|\le2p-2$.

For a contradiction, suppose that $|A_L|\le2p-3$ for every $L$.
Since $3(2p-3)>4p-4$, there are at most two subspaces
with $|A_L|=2p-3$.
For a term $\mathbf z$ whose residue modulo $p$ lies in a subspace
with largest $|A_L|$, let $c_L$ count the remaining terms
whose residues modulo $p$ lie in $L$.
Then $\sum_Lc_L=4p-5$ and $c_L\le2p-3$,
with equality for at most one $L$. 
In the following, we partition $A=J_1J_2\mathbf z$ 
with $|J_1|=2p-2$ and $|J_2|=2p-3$, 
and show that $A$ is not $w$-bad, a contradiction.  

{\textbf{Case 1: $p=3$.}}
For a subspace $L$ with largest count, we have $c_L\in\{2,3\}$.
If $c_L=3$, then $J_2$ consists of one term whose residue lies
in $L$ and one from each of two other nonempty residue groups.
If $c_L=2$, the other three counts are $(2,2,1)$, and $J_2$
consists of one term from each of these three groups.
In either case, the remaining four terms form $J_1=BC$,
where $B$ consists of two terms whose residues lie in $L$ 
and $C$ consists of two terms whose residues lie outside $L$.
Thus, $A=J_1J_2\mathbf z$ and $|B|=|C|=2$. 
Since both $L$ and $\F_3^2/L$ are cyclic groups of order $3$,
it follows from Lemma~\ref{lem:cyclic-growth} that
$\Sigma(B)=L$ and $\Sigma(C)$ maps onto $\F_3^2/L$. 
Hence, $\Sigma(J_1)=\F_3^2$.
Since the three terms of $J_2$ have residues in distinct
one-dimensional subspaces, it follows from Lemma~\ref{lem:two-missing} that 
$|\F_3^2\setminus\Sigma(J_2)|\le2$.

{\textbf{Case 2: $p\ge5$.}}
For each $L$, we assign $\lfloor \frac{c_L}{2}\rfloor$ terms to each
subsequence.
The remaining terms from groups with odd $c_L$ are assigned
alternately to $J_1$ and $J_2$, starting with $J_1$ and placing
$c_L=2p-3$ first if it occurs.
Thus, $A=J_1J_2\mathbf z$.
Since $\sum_Lc_L=4p-5$ is odd, there are an odd number of
remaining terms and $|J_1|-|J_2|=1$.
Together with $|J_1|+|J_2|=4p-5$, it forces 
$|J_1|=2p-2$ and $|J_2|=2p-3$.
By construction, each one-dimensional subspace contains
the residues of at most $p-1$ terms of $J_1$ and at most
$p-2$ terms of $J_2$. 
By Lemmas~\ref{lem:matomaki} and~\ref{lem:two-missing},
we therefore have $|\F_p^2\setminus\Sigma(J_1)|\le1$ and
$|\F_p^2\setminus\Sigma(J_2)|\le2$, respectively.

Since $\mathbf z\bmod p\ne\mathbf0$, the subgroup
$\langle\mathbf z\bmod p\rangle$ has $p$ elements.
In either case, we conclude that
\[
 |\F_p^2\setminus\Sigma(J_1)|+|\F_p^2\setminus\Sigma(J_2)|
 <p=[\F_p^2:~\langle\mathbf z\bmod p\rangle].
\]
By Lemma~\ref{lem:two-complete-blocks}(ii),
$A$ is not $w$-bad, a contradiction.
There is therefore a subspace $L$ with $|A_L|=2p-2$.

We next determine the structure of $A$ modulo $p^2$.
With $A'=A\setminus A_L$ and $L=\{(t,0):~t\in\F_p\}$,
we have $A=A_LA'$ and $|A_L|=|A'|=2p-2$.
Recall that every term of $A_L$ has first coordinate nonzero
and second coordinate zero modulo $p$, and that
\eqref{eq:first-coordinate-lift} holds for every nonempty
subsequence $T$ of $A_L$. 

For $A'$, every term has second coordinate nonzero modulo $p$,
since its residue lies outside $L$.
For any nonempty subsequence $T$ of $A'$ whose second-coordinate
sum satisfies $\sigma(T)_2\equiv0\pmod p$,
suppose, for a contradiction, that
$\sigma(T)_1\not\equiv0\pmod p$.
Then the first coordinates of $A_L$, together with the appended
term $\sigma(T)_1$, form an admissible sequence in $\Z$
of length $2p-1$.
For any nonempty subsequence of this integer sequence with sum
zero modulo $p$, let $R$ consist of the terms of $A_L$
whose first coordinates are selected.
The resulting subsequence of $A$ is $R$ if the appended term
is not selected, and $RT$ otherwise.
In either case, it is nonempty and has both coordinate sums
zero modulo $p$, that is, its vector sum is zero modulo $p^2$
by the $w$-badness of $A$. 
Thus, the integer sequence is $w$-bad, contrary to the former fact $w_1(p)=2p-1$.
Consequently, $\sigma(T)_1\equiv0\pmod p$, and the $w$-badness
of $A$ gives, similarly to~\eqref{eq:first-coordinate-lift},   
\begin{equation}\label{eq:second-coordinate-lift}
 \sigma(T)_2\equiv0\pmod p
 \ \Longrightarrow\ \sigma(T)\equiv\mathbf0\pmod p
 \ \Longrightarrow\ \sigma(T)\equiv\mathbf0\pmod{p^2}.
\end{equation}

Applying Lemma~\ref{lem:sun-scalar-classification} to the first
coordinates of $A_L$ and the second coordinates of $A'$,
we obtain the coordinate values $\pm u$ and $\pm v$, respectively,
modulo $p^2$.
By~\eqref{eq:first-coordinate-lift}, any two terms of $A_L$
with opposite first coordinates sum to $\mathbf0$ modulo $p^2$.
Hence, all terms with first coordinate $u$ equal a fixed vector
$\mathbf U$ modulo $p^2$, and those with first coordinate $-u$
equal $-\mathbf U$.
Similarly, by~\eqref{eq:second-coordinate-lift}, the terms of $A'$
with second coordinates $v$ and $-v$ equal $\mathbf V$ and
$-\mathbf V$, respectively, modulo $p^2$. 
Thus,
\[
 A=A_LA'
 \equiv\mathbf U^{[p-1]}(-\mathbf U)^{[p-1]}
       \mathbf V^{[p-1]}(-\mathbf V)^{[p-1]}\pmod{p^2},
\]
where $\mathbf U=(U_1,U_2)$ and $\mathbf V=(V_1,V_2)$ satisfy
$U_2\equiv0\pmod p$ and $U_1,V_2\not\equiv0\pmod p$.
Hence,
$
 \det(\mathbf U,\mathbf V)
 =U_1V_2-U_2V_1
 \equiv U_1V_2\not\equiv0\pmod p.
$ 
This proves necessity. 

For \emph{sufficiency}, after reordering, we have
$A\equiv (\mathbf U,\mathbf V) E_{2,p}\pmod{p^2}$, where
\[
 E_{2,p}
 =\mathbf e_1^{[p-1]}(-\mathbf e_1)^{[p-1]}
  \mathbf e_2^{[p-1]}(-\mathbf e_2)^{[p-1]} 
\] is as in Lemma~\ref{lem:axis}  
and  $(\mathbf U,\mathbf V) E_{2,p}$ denotes the sequence obtained by applying 
$(\mathbf U,\mathbf V)$ to each term of $E_{2,p}$.
Since $\det(\mathbf U,\mathbf V)\not\equiv0\pmod p$, the matrix $(\mathbf U,\mathbf V)$ is invertible
modulo $p^2$.
It follows directly from Lemma~\ref{lem:axis} that $A$ is
admissible and $w$-bad.
This completes the proof.
\end{proof}

\begin{corollary}\label{cor:prime-exact}
For every prime $p$, we have $w_2(p)=4p-3$.
\end{corollary}
\begin{proof}
When $p=2$, the result follows from Theorem~\ref{thm:binary} that $w_2(2)=5$.  

For an odd prime $p$, suppose that there exists an admissible sequence $A$ with $|A|=4p-3$, 
which is $w$-bad modulo $p$.
For any term $\mathbf z$ of $A$, it follows from
Theorem~\ref{lem:prime-w2} that $A\setminus\mathbf z$ has four
residue classes modulo $p$, each containing $p-1$ terms.
There is therefore a term $\mathbf x$ of $A\setminus\mathbf z$
with $\mathbf x\not\equiv\mathbf z\pmod p$ and the admissible sequence 
$A\setminus\mathbf x$ is still $w$-bad modulo $p$. 
However, the residue class of $\mathbf x$ only contains exactly $p-2$ terms in $A\setminus\mathbf x$, 
contrary to Theorem~\ref{lem:prime-w2} since $1\le p-2<p-1$.
Therefore, there is no admissible $w$-bad sequence of length
$4p-3$, yielding $$w_2(p)\le4p-3.$$ 
 
By taking $n=p$ in Lemma~\ref{lem:axis}, we also obtain the lower bound $w_2(p)\ge4p-3$. 
As a result, $w_2(p)=4p-3$ holds for any prime $p$.
\end{proof}

\subsection{Proofs of Theorems~\ref{thm:exact}(ii) and \ref{thm:exact}(iii)}
We now prove Theorem~\ref{thm:exact}(ii) and then combine it
with Corollary~\ref{cor:prime-exact} to obtain part~(iii). 
For the part of the tripling inequality in Theorem~\ref{thm:exact}(ii), 
we also need the following length bound
and the structure of sequences attaining it.

\begin{lemma}\label{lem:ternary-short}
Let $A$ be an admissible sequence in $\Z^2$ for $n=3$.
Suppose that every nonempty subsequence $B$ of $A$ with $2\leq |B|\le3$
and $\sigma(B)\equiv\mathbf0\pmod3$ satisfies
$\sigma(B)\equiv\mathbf0\pmod9$.
Then $|A|\le12$. 
In particular, if $|A|=12$, there are $\mathbf U,\mathbf V\in\Z^2$
with $\det(\mathbf U,\mathbf V)\not\equiv0\pmod3$ such that,
up to reordering,
$$
 A\equiv\mathbf U^{[2]}(-\mathbf U)^{[2]}
 \mathbf V^{[2]}(-\mathbf V)^{[2]}
 (\mathbf U+\mathbf V)^{[2]}(-\mathbf U-\mathbf V)^{[2]}
 \pmod9.
$$
\end{lemma}
\begin{proof}
The bound is immediate if $|A|<12$, so we may assume that
$|A|\ge12$. 
Let
\[
 S=\{\mathbf a\bmod3:~\mathbf a\text{ is a term of }A\}
 \subseteq\F_3^2\setminus\{\mathbf0\}.
\]
We first show that each element of $S$ is represented by at most
two terms of $A$.

Suppose that four terms $\mathbf x_1,\ldots,\mathbf x_4$ of $A$ have
the same residue $\mathbf u\in S$ modulo $3$. 
Then $\sum_{j\ne i}\mathbf x_j\equiv\mathbf0\pmod9$
for $1\le i\le4$. 
By subtraction, we have $\mathbf x_1\equiv\mathbf x_2\equiv\mathbf x_3\equiv\mathbf x_4\pmod9$ 
and hence, $3\mathbf x_1\equiv\mathbf0\pmod9$, contrary to admissibility. 
Thus, each element of $S$ is represented by at most three terms of $A$.

Suppose that three terms $\mathbf x_1,\mathbf x_2,\mathbf x_3$ of $A$
have the same residue $\mathbf u\in S$ modulo $3$. 
If a subsequence $R$ of $A$, disjoint from these three terms,
satisfies $1\le|R|\le2$ and $\sigma(R)\equiv-\mathbf u\pmod3$,
then
$\mathbf x_1\equiv\mathbf x_2\equiv\mathbf x_3
\equiv-\sigma(R)\pmod9$.  
Again, we have $3\mathbf x_1\equiv\mathbf0\pmod9$, contrary to admissibility. 
Consequently, $-\mathbf u\notin S$.
For any $\mathbf v\in\F_3^2\setminus\langle\mathbf u\rangle$,
the set $\F_3^2\setminus\{\mathbf0,\mathbf u,-\mathbf u\}$
can be partitioned into the three pairs
$
 \{\mathbf v,-\mathbf u-\mathbf v\},$ 
 $\{-\mathbf v,-\mathbf u+\mathbf v\},$ and 
$\{\mathbf u+\mathbf v,\mathbf u-\mathbf v\}.
$
Since each pair $\{\mathbf a,\mathbf b\}$ satisfies
$\mathbf a+\mathbf b=-\mathbf u$, 
$|S\cap\{\mathbf a,\mathbf b\}|\le1$. 
Together with $\mathbf u\in S$, it forces 
$
 |S|=1+|S\setminus\{\mathbf u\}|\le1+3=4.
$ 
Then 
$
 12\le|A|\le3|S|\le12.
$
Therefore, $|S|=4$ and each element of $S$ is represented
by exactly three terms of $A$.
The exclusion of $R$ applies to each of these residues, giving 
\begin{align}\label{eq.lem39}
S\cap(-S)=\varnothing,\qquad 
-\mathbf a-\mathbf b\notin S,~\forall~ \mathbf a,\mathbf b\in S~{\rm with}~\mathbf a\ne\mathbf b.  
\end{align}
As a result, all of the four elements of $S$ lie in distinct
one-dimensional subspaces of $\F_3^2$.
For distinct $\mathbf u,\mathbf v\in S$, the other two elements
belong respectively to $\{\pm(\mathbf u+\mathbf v)\}$ and
$\{\pm(\mathbf u-\mathbf v)\}$.
Since $-\mathbf u-\mathbf v\notin S$, after interchanging
$\mathbf u,\mathbf v$ if necessary, we have
$S=\{\mathbf u,\mathbf v,\mathbf u+\mathbf v,\mathbf u-\mathbf v\}$.
However,
$
 -(\mathbf u+\mathbf v)-(\mathbf u-\mathbf v)=\mathbf u\in S,
$
contrary to the condition in \eqref{eq.lem39}. 
This contradiction shows that each element of $S$ is represented
by at most two terms of $A$. 

We now have that $|S|\ge|A|/2\ge6$. 
Since the total number of one-dimensional subspaces of $\F_3^2$ is four, 
there are at least two one-dimensional subspaces of $\F_3^2$ that contain both nonzero
elements in $S$.
There are therefore terms $\mathbf U,\mathbf V$ of $A$ with
linearly independent residues
$\mathbf u=\mathbf U\bmod3$ and $\mathbf v=\mathbf V\bmod3$
such that $\{\pm\mathbf u,\pm\mathbf v\}\subseteq S$.
Since $|S|\ge6$, at least one of
$\pm\mathbf u\pm\mathbf v$ belongs to $S$.
After changing the signs of $\mathbf u,\mathbf v$ and choosing
terms $\mathbf U,\mathbf V$ with these residues, we may assume
that $\mathbf u+\mathbf v\in S$.

Under the hypothesis required for the lemma, 
by applying the hypothesis to two-term subsequences, as in the
proof of Theorem~\ref{lem:prime-w2}, and then to three-term
subsequences, we obtain, for every term $\mathbf x$ of $A$
and $\varepsilon\in\{1,-1\}$,
\begin{align}\label{eq.lem3922}
 \begin{aligned}
 \mathbf x\bmod3=\varepsilon\mathbf u
 &\ \Longrightarrow\
 \mathbf x\equiv\varepsilon\mathbf U\pmod9,\\
 \mathbf x\bmod3=\varepsilon\mathbf v
 &\ \Longrightarrow\
 \mathbf x\equiv\varepsilon\mathbf V\pmod9,\\
 \mathbf x\bmod3=\varepsilon(\mathbf u+\mathbf v)
 &\ \Longrightarrow\
 \mathbf x\equiv\varepsilon(\mathbf U+\mathbf V)\pmod9.
 \end{aligned}
\end{align}
If $S\nsubseteq\{\pm\mathbf u,\pm\mathbf v,
\pm(\mathbf u+\mathbf v)\}$,
then, after interchanging $\mathbf U,\mathbf V$ if necessary,
$\mathbf u-\mathbf v\in S$.
There are therefore terms $\mathbf W,\mathbf Q$ of $A$ satisfying
$\mathbf W\bmod3=\mathbf u+\mathbf v$ and
$\mathbf Q\bmod3=\mathbf u-\mathbf v$.
The hypothesis applied to $\mathbf Q$ and terms of $A$
with residues $-\mathbf u,\mathbf v$ gives
$\mathbf Q-\mathbf U+\mathbf V\equiv\mathbf0\pmod9$.
The three-term subsequence $\mathbf W\mathbf Q\mathbf U$ of $A$
has sum zero modulo $3$, but 
$
 \mathbf W+\mathbf Q+\mathbf U
 \equiv3\mathbf U\not\equiv\mathbf0\pmod9,
$ 
a contradiction to the hypothesis. 
It follows that
$S\subseteq\{\pm\mathbf u,\pm\mathbf v,\pm(\mathbf u+\mathbf v)\}$
and hence $|A|\le2|S|\le12$.
This proves the first assertion.

If $|A|=12$, equality holds throughout, so $|S|=6$
and each element of $S$ is represented by exactly two terms of $A$.
The congruences~\eqref{eq.lem3922} give the stated representation,
with $\det(\mathbf U,\mathbf V)\not\equiv0\pmod3$.
This completes the proof. 
\end{proof}

\begin{proof}[\textbf{Proof of Theorem~\ref{thm:exact}(ii)}]
For a contradiction, suppose that $A$ is an admissible
$w$-bad sequence of length $dw_2(t)+3(d-1)$ in $\Z^2$ modulo $dt$.
Consider the following two types of nonempty subsequences $B$ of $A$:
\[
 \begin{aligned}
 \text{Type I:}\quad
 &|B|=1,\qquad \sigma(B)\equiv\mathbf0\pmod d;\\
 \text{Type II:}\quad
 &\sigma(B)\equiv\mathbf0\pmod d,\qquad
  \sigma(B)\not\equiv\mathbf0\pmod{d^2}.
 \end{aligned}
\]
For any pairwise disjoint subsequences $B_1,\ldots,B_m$ of $A$,
each of Type~I or Type~II, we claim that $m<w_2(t)$.
For Type~I and Type~II,
$\sigma(B_i)\not\equiv\mathbf0\pmod{dt}$
follows from admissibility and $w$-badness of $A$, respectively.
Hence, each $\mathbf c_i=\sigma(B_i)/d$ is nonzero modulo $t$,
thereby showing that $\mathbf c_1\cdots\mathbf c_m$
is admissible modulo $t$.
For every nonempty $I\subseteq[m]$, the $w$-badness of $A$
implies that, if $\sum_{i\in I}\mathbf c_i\equiv\mathbf0\pmod t$,
then
$\sigma\left(\prod_{i\in I}B_i\right)
=d\sum_{i\in I}\mathbf c_i
\equiv\mathbf0\pmod{d^2t^2}$,
and hence $\sum_{i\in I}\mathbf c_i\equiv\mathbf0\pmod{t^2}$.
Thus, $\mathbf c_1\cdots\mathbf c_m$ is $w$-bad modulo $t$ 
and 
\begin{equation}\label{eq.thm1.2ii}
 m<w_2(t).
\end{equation}

To contradict~\eqref{eq.thm1.2ii}, we now construct
$w_2(t)$ pairwise disjoint nonempty subsequences of $A$,
each of Type~I or Type~II.
Let $q$ denote the number of terms of $A$ that are zero modulo $d$.
These terms give $q$ pairwise disjoint one-term subsequences of Type~I, 
and hence, $q<w_2(t)$ by \eqref{eq.thm1.2ii}.
Note that all the other $|A|-q=dw_2(t)+3(d-1)-q$ terms of $A$
are nonzero modulo $d$.
We have two cases, depending on the value of $d$.

\textbf{Case 1: $d=2$.}
Among any seven remaining terms, three terms
$\mathbf x,\mathbf y,\mathbf z$
have the same nonzero residue modulo $2$.
Since
$2\mathbf x
=(\mathbf x+\mathbf y)+(\mathbf x+\mathbf z)
-(\mathbf y+\mathbf z)
\equiv\mathbf0\pmod4$
cannot be allowed, at least one of the three pair sums
$\mathbf x+\mathbf y,\mathbf x+\mathbf z,\mathbf y+\mathbf z$
is nonzero modulo $4$.
This gives a two-term subsequence of Type~II.

After deleting the $q$ terms of $A$ that are zero modulo $2$,
there remain $2w_2(t)+3-q$ terms.
If $j$ disjoint pairs of Type~II have been selected and removed
from these terms, the remaining subsequence of $A$ has length
$2w_2(t)+3-q-2j$.
For $0\le j<w_2(t)-q-1$, this length is at least $7+q\ge7$,
which means another pair of Type~II can be selected.
Thus, $w_2(t)-q-1$ disjoint pairs can be selected successively,
leaving $5+q$ terms of $A$.
Since $w_2(2)=5$ by Theorem~\ref{thm:exact}(i),
these remaining terms contain a further subsequence of Type~II. 
Thus, we obtain 
$$
 q+(w_2(t)-q-1)+1=w_2(t)
$$
pairwise disjoint nonempty subsequences of $A$,
each of Type~I or Type~II,
contradicting~\eqref{eq.thm1.2ii}.

\textbf{Case 2: $d=3$.}
By Corollary~\ref{cor:prime-exact}, we have $w_2(3)=9$. 
After deleting the $q$ terms of $A$ that are zero modulo $3$,
there remain $3w_2(t)+6-q$ terms, all nonzero modulo $3$.
If $w_2(t)-q=1$, the remaining length is
$9+2q\ge9=w_2(3)$ and these terms contain a subsequence of Type~II.
Together, they give $q+1=w_2(t)$ pairwise disjoint nonempty
subsequences of $A$, each of Type~I or Type~II,
contradicting~\eqref{eq.thm1.2ii}.

If $w_2(t)-q\ge2$, the remaining length is at least $12+2q$.
Whenever at least $13$ terms remain,
Lemma~\ref{lem:ternary-short} gives a subsequence of Type~II
with at most three terms.
By an argument similar to that in \textbf{Case~1}, we can select and remove
$w_2(t)-q-2$ such subsequences of Type~II from the remaining terms,
leaving at least $12+2q$ terms.
Let $J$ consist of any $12$ of these remaining terms.

\begin{itemize}
\item
If $J$ contains a subsequence $C$ of Type~II with $|C|\le3$,
then $|J\setminus C|\ge9=w_2(3)$.
Hence, $J\setminus C$ contains another subsequence $D$
of Type~II, disjoint from $C$.

\item
Otherwise, every nonempty subsequence of $J$ with at most three terms 
and sum zero modulo $3$ has sum zero modulo $9$,
which satisfies the hypothesis of Lemma~\ref{lem:ternary-short}.
From Lemma~\ref{lem:ternary-short}, we can get
two disjoint subsequences $C,D$ of $J$ such that
$C\equiv
 \mathbf U^{[2]}(\mathbf U+\mathbf V)(-\mathbf V)
 \pmod9
 $ and $D\equiv
 (-\mathbf U)^{[2]}(-\mathbf U-\mathbf V)\mathbf V
 \pmod9.
$
Since their respective sums are $3\mathbf U$ and $-3\mathbf U$
modulo $9$, and $\mathbf U\not\equiv\mathbf0\pmod3$,
both $C$ and $D$ are of Type~II.
\end{itemize}
In either subcase, since $J$ consists of remaining terms of $A$,
both $C$ and $D$ are disjoint from all previously selected
subsequences. We therefore obtain 
$$q+(w_2(t)-q-2)+2=w_2(t)$$ 
pairwise disjoint nonempty subsequences of $A$, each of Type~I or Type~II,
contradicting~\eqref{eq.thm1.2ii}.

As a result, there is no admissible $w$-bad sequence of length
$dw_2(t)+3(d-1)$ in $\Z^2$ modulo $dt$, and hence, $w_2(dt)\le dw_2(t)+3(d-1)$.
This completes the proof of Theorem~\ref{thm:exact}(ii). 
\end{proof}

\begin{proof}[\textbf{Proof of Theorem~\ref{thm:exact}(iii)}]
By Corollary~\ref{cor:prime-exact}, we have $w_2(p)=4p-3$
for every prime $p$.
By successive applications of Theorem~\ref{thm:exact}(ii) 
for any integers $k\geq 0$ and $\ell\geq 0$ and any prime $p$, we obtain
\[
 w_2(2^k3^\ell p)
 \le2^k3^\ell\bigl(w_2(p)+3\bigr)-3
 =4\cdot2^k3^\ell p-3.
\]
Together with the lower bound in Lemma~\ref{lem:axis},
this proves Theorem~\ref{thm:exact}(iii). 
\end{proof}

We end this section with the following remark and question. 

\begin{remark}
The lower bound $w_r(n)\ge2r(n-1)+1$ in Lemma~\ref{lem:axis}
is attained for 
\begin{itemize}
    \item $r=1$ and $n>1$ by Sun's result~\cite[Theorem~1.1]{S26},
    \item $r\geq 1$ and $n=2$  by Theorem~\ref{thm:exact}(i), and
    \item $r=2$ and $n=2^k3^\ell p$ by Theorem~\ref{thm:exact}(iii), where $p$ is a prime and $k,\ell\ge0$ are integers,
\end{itemize}
This leads to the following question. 

\begin{question}
Does $w_r(n)=2r(n-1)+1$ hold for every $r\ge1$ and $n>1$?
\end{question}
\end{remark}

\section{Proof of Theorem~\ref{thm:all-moduli}}\label{sec:all-moduli}

This section proves Theorem~\ref{thm:all-moduli}. 
We first bound $w_r(n)$ by the maximum length of a reduced sequence
and then estimate this length. 
Following Letzter and Morrison~\cite[Section~2]{LM24}, a sequence
$S$ over a finite abelian group $G$ is \emph{reduced} if
$|\Sigma(S)|>|\Sigma(S\setminus z)|$ for every $z\in S$,
where $S\setminus z$ is obtained from $S$ by deleting one copy of $z$. 
As stated in \cite[Observation~2.2]{LM24}, every subsequence of a reduced sequence is again reduced.
Let $h(G)$ be the largest length of a reduced sequence over $G$. 
For every reduced sequence $S$, we have $|\Sigma(S)|\ge|S|+1$, and hence $h(G)\le|G|-1$. 

The following lemma means that every admissible $w$-bad sequence 
can be split into two reduced subsequences. 

\begin{lemma}\label{lem:two-reduced}
Let $n>1$ and $r\ge1$. Every admissible $w$-bad sequence
$A$ in $\Z^r$ can be written as $A=BC$, where the residue
sequences of both $B$ and $C$ modulo $n$ are reduced.
Moreover, $$w_r(n)\le2h(\Z_n^r)+1.$$ 
\end{lemma}
\begin{proof}
Choose a subsequence $B$ of $A$ of shortest length with
$\Sigma(B)=\Sigma(A)$. Then its residue sequence modulo $n$ is reduced.
For $C=A\setminus B$, suppose that its residue sequence is not reduced.
Then $\Sigma(C\setminus\mathbf z)=\Sigma(C)$ for some specified
term $\mathbf z$ of $C$. 
By the identity
$\Sigma(C)=\Sigma(C\setminus\mathbf z)\cup
(\Sigma(C\setminus\mathbf z)+(\mathbf z\bmod n))$, 
we obtain $\Sigma(C)+(\mathbf z\bmod n)\subseteq\Sigma(C)$.
Since the two sets have the same cardinality, 
$\Sigma(C)+(\mathbf z\bmod n)=\Sigma(C)$.
Since $\mathbf0\in\Sigma(C)$, we have 
$H=\langle\mathbf z\bmod n\rangle\subseteq\Sigma(C)$.
Together with the facts that $-H=H$ and $\Sigma(C)\subseteq\Sigma(B)$, 
we obtain 
$$
 H\subseteq\Sigma(B)\cap\bigl(-\Sigma(C\setminus\mathbf z)\bigr).
$$
By the coset condition in Lemma~\ref{lem:two-complete-blocks}(i), 
we deduce that the admissible sequence 
$A=B(C\setminus\mathbf z)\mathbf z$ is not $w$-bad, a contradiction.
Thus, both residue sequences of $B$ and $C$ are reduced. 
Moreover, $|A|=|B|+|C|\le2h(\Z_n^r)$.
Since $A$ was arbitrary, $w_r(n)\le2h(\Z_n^r)+1$. 
\end{proof}

\begin{lemma}{\rm (\!\! \cite[Corollary~1.3]{A21})}\label{lem:arsovski-additive-basis}
Let $G$ be a finite abelian group of exponent $m$, and let $T$
be a sequence over $G$. If every subsequence of $T$ of length at least
$|T|-m\ln|G|$ generates $G$, then $\Sigma(T)=G$. 
\end{lemma}

\begin{corollary}\label{lem:quadratic-reduced}
Let $G$ be a finite abelian group whose exponent divides $n\ge2$.
For $d=\lceil\log_2|G|\rceil$, we have
$$h(G)\le\frac{nd(d+1)}{2}.$$
\end{corollary}
\begin{proof}
We argue by induction on $|G|$ for fixed $n$.
Since a reduced sequence cannot contain zero, the result is immediate
for the trivial group.
For a contradiction, suppose that $S$ is a reduced sequence over a
nontrivial group $G$ with $|S|>\frac{nd(d+1)}2$.
For any term $z$ of $S$, the sequence $T=S\setminus z$ then has length 
$|T|\ge\frac{nd(d+1)}2$.

For every proper subgroup $H<G$, the terms of $T$ lying in $H$
form a reduced sequence.
Since $\lceil\log_2|H|\rceil\le d-1$, the induction hypothesis
bounds their number by $\frac{nd(d-1)}2$.
Thus, the number of terms of $T$ outside $H$ is at least
$$
 |T|-\frac{nd(d-1)}2\ge nd>n\ln|G|.
$$
Consequently, every subsequence of $T$ of length at least
$|T|-n\ln|G|$ generates $G$.
Since the exponent of $G$ divides $n$, 
it turns out from Lemma~\ref{lem:arsovski-additive-basis} that 
$\Sigma(T)=G=\Sigma(S)$, a contradiction. 
This completes the induction and proves the corollary. 
\end{proof}

\begin{proof}[\textbf{Proof of Theorem~\ref{thm:all-moduli}}]
Taking $G=\Z_n^r$ in Lemma~\ref{lem:two-reduced}
and Corollary~\ref{lem:quadratic-reduced}, we obtain 
\begin{equation}\label{eq:coarse-w}
 w_r(n)\le2h(\Z_n^r)+1\le n\lceil r\log_2 n\rceil(\lceil r\log_2 n\rceil+1)+1.
\end{equation}

Note that the classical lower bound~\cite{H73} for the Erd\H{o}s-Ginzburg-Ziv constant $s(\Z_n^r)$ 
is 
\begin{equation}\label{eq:cube-egz}
 s(\Z_n^r)\ge2^r(n-1)+1.
\end{equation}
For fixed $n\ge2$, inequalities~\eqref{eq:coarse-w}
and~\eqref{eq:cube-egz} give
$$
\begin{aligned}
0\le \frac{w_r(n)}{s(\Z_n^r)}
&\le
\frac{n\lceil r\log_2 n\rceil
(\lceil r\log_2 n\rceil+1)+1}{2^r(n-1)+1}\\
&\le
\frac{n(r\log_2 n+1)(r\log_2 n+2)+1}{2^r(n-1)}\\
&=
\frac{n(\log_2 n)^2}{n-1}\cdot \frac{r^2}{2^r}
+\frac{3n\log_2 n}{n-1}\cdot \frac{r}{2^r}
+\frac{2n+1}{n-1}\cdot \frac{1}{2^r}\\
&=O_n\left(\frac{r^2}{2^r}\right). 
\end{aligned}
$$
It implies that $\displaystyle\lim_{r\to\infty}
w_r(n)/s(\Z_n^r)=0$. 

We next prove $w_r(n)<s(\Z_n^r)$ for the stated ranges of $r$. 
For $n=2$, using Theorem~\ref{thm:exact}(i), it is direct to verify that 
$w_r(2)=2r+1<2^r+1=s(\Z_2^r)$ for $r\ge3$, where the last equality is due to~\cite{H73}. 
For $n\ge3$, suppose that
$r\ge\lceil4\log_2(1+\log_2 n)+4\rceil$.
By the elementary inequality $\log_2 x\le x/2+1/4$ for $x>0$, we have 
$$
\begin{aligned}
(\log_2 n)\bigl(4\log_2(1+\log_2 n)+4\bigr)+2
\le 2(\log_2 n)^2+7\log_2 n+2 
<3(1+\log_2 n)^2.
\end{aligned}
$$
Since $(x\log_2 n+1)(x\log_2 n+2)/2^x$ is decreasing in $x$ for $x\ge3$,
we obtain
$$
\begin{aligned}
\frac{\lceil r\log_2 n\rceil(\lceil r\log_2 n\rceil+1)}{2^r}
\le \frac{\bigl[(\log_2 n)
(4\log_2(1+\log_2 n)+4)+2\bigr]^2}
{16(1+\log_2 n)^4}
<\frac{9}{16}<\frac{2}{3} \le 1-\frac{1}{n}.
\end{aligned}
$$
Therefore, we deduce from inequalities~\eqref{eq:coarse-w} and~\eqref{eq:cube-egz} 
that 
$$
\begin{aligned}
w_r(n)
\le n\lceil r\log_2 n\rceil(\lceil r\log_2 n\rceil+1)+1 
<2^r(n-1)+1
\le s(\Z_n^r).
\end{aligned}
$$
This completes the proof of Theorem~\ref{thm:all-moduli}. 
\end{proof}

\begin{remark}
For prime moduli, a more precise estimate for $w_r(p)$ can be obtained by combining
Lemma~\ref{lem:two-reduced} with the bound $h(\Z_p^r)\le(p-1)r(\log_2r+1)$ shown in 
\cite[Theorem~2.1]{LM24}, where $p$ is a prime and $r\ge1$. 
By an argument similar to that in the proof of Theorem~\ref{thm:all-moduli}, 
we specifically obtain 
$$w_r(p)\le2(p-1)r(\log_2r+1)+1$$ and 
$$\frac{w_r(p)}{s(\Z_p^r)}=O_p\left(\frac{r\log r}{2^r}\right)\to 0,~r\to\infty.$$ 
\end{remark}

\section*{Acknowledgments}
OpenAI’s ChatGPT (GPT-6 Astra and GPT-6.1 Sol) was used for improving the presentation of the paper. 
The research is supported by 
the Nanyang Technological University Research Grant under No. 04INS000047C230GRT01.

\end{document}